\documentclass[11pt]{amsart}

\usepackage{amsmath,amssymb,latexsym,amsthm,newlfont,enumerate}

\usepackage{graphicx}
\usepackage[all]{xy}
\makeindex

\theoremstyle{plain}

\newtheorem{thm}{Theorem}[section]

\newtheorem{lem}[thm]{Lemma}

\newtheorem{cla}[thm]{Claim}
\newtheorem{cor}[thm]{Corollary}

\theoremstyle{definition}

\newtheorem{dfn}[thm]{Definition}

\newtheorem{rem}[thm]{Remark}

\newtheorem{exa}[thm]{Example}

\theoremstyle{remark}

\newcommand{\Z}{\mathbb{Z}}
\newcommand{\N}{\mathbb{N}}
\newcommand{\C}{\mathbb{C}}
\newcommand{\R}{\mathbb{R}}
\newcommand{\Q}{\mathbb{Q}}
\newcommand{\PS}{\mathbb{P}}

\newcommand{\OO}{\mathcal{O}}

\newcommand{\vphi}{\varphi}

\newcommand{\dashto}{\dashrightarrow}
\newcommand{\lto}{\longrightarrow}
\newcommand{\mcal}{\mathcal}

\DeclareMathOperator{\Bs}{Bs}

\DeclareMathOperator{\mult}{mult}

\DeclareMathOperator{\Supp}{Supp}

\DeclareMathOperator{\Fix}{Fix}

\DeclareMathOperator{\Proj}{Proj}

\DeclareMathOperator{\Mob}{Mob}
\DeclareMathOperator{\Pic}{Pic} 

\DeclareMathOperator{\Div}{Div}

\DeclareMathOperator{\Nef}{Nef}
\DeclareMathOperator{\Amp}{Amp}
\DeclareMathOperator{\Effb}{\overline{\mathrm{Eff}}}
\DeclareMathOperator{\Eff}{\mathrm{Eff}}
\DeclareMathOperator{\Mov}{Mov}
\DeclareMathOperator{\B}{Big}

\DeclareMathOperator{\ddiv}{div}
\DeclareMathOperator{\sB}{\mathbf{B}}

\begin{document}
\title[Finite generation and geography of models]{Finite generation and geography of models}

\author{Anne-Sophie Kaloghiros}
\address{Department of Mathematics, Imperial College London, 180 
Queen's Gate, London SW7 2AZ, UK}
\email{a.kaloghiros@imperial.ac.uk}
\author{Alex K\"{u}ronya}
\address{Mathematisches Institut, Albert-Ludwigs-Universit\"at Freiburg, Eckerstra\ss e 1, 79104 Freiburg, Germany}
\address{Mathematical Institute, Budapest University of Technology and Economics, P.O. Box 91, 1521 Budapest, Hungary}
\email{alex.kueronya@math.uni-freiburg.de}
\author{Vladimir Lazi\'c}
\address{Mathematisches Institut, Universit\"at Bayreuth, 95440 Bayreuth, Germany}
\email{vladimir.lazic@uni-bayreuth.de}

\dedicatory{To Professor Shigefumi Mori on the occasion of his 60th birthday, with admiration}
\thanks{The first author was supported by the Engineering and Physical Sciences Research Council [grant number EP/H028811/1]. The second author was supported by the DFG-Forschergruppe 790 ``Classification of Algebraic Surfaces and Compact Complex Manifolds", and by the OTKA Grants 77476 and 81203 of the Hungarian Academy of Sciences. The third author was supported by the DFG-Forschergruppe 790 ``Classification of Algebraic Surfaces and Compact Complex Manifolds". We would like to thank P.~Cascini, A.-M.~Castravet, Y.~Gongyo, K.~Matsuki and the referee for many useful comments, and the following institutions for providing additional support and excellent working conditions: Albert-Ludwigs-Universit\"at Freiburg, Imperial College London, University of Illinois at Chicago.}

\begin{abstract}
There are two main examples where a version of the Minimal Model Program can, at least conjecturally, be performed successfully: the first is the classical MMP associated to the canonical divisor, and the other is Mori Dream Spaces. In this paper we formulate a framework which generalises both of these examples. Starting from divisorial rings which are finitely generated, we determine precisely when we can run the MMP, and we show why finite generation alone is not sufficient to make the MMP work.
\end{abstract}

\maketitle
\setcounter{tocdepth}{1}
\tableofcontents

\section{Introduction}

There are two classes of projective varieties whose birational geometry is particularly interesting and rich. The first family consists of varieties where the classical Minimal Model Program (MMP) can be performed successfully with the current techniques. The other class is that of Mori Dream Spaces.  We now know that, in both cases, the geometry of birational contractions from the varieties in question is entirely determined by suitable finitely generated divisorial rings. 

More precisely, let $X$ be a $\Q$-factorial projective variety that belongs to one of these two classes. Then, there are effective $\Q$-divisors $D_1,\dots,D_r$ strongly related to the geometry of $X$ such that the multigraded divisorial ring 
$$ \mathfrak R= R(X; D_1, \dots, D_r)= \bigoplus_{(n_1, \dots , n_r)\in \N^r} H^0(X, n_1D_1+\dots+n_rD_r)$$
is finitely generated. In the first case, $\mathfrak R$ is an adjoint ring; in the second, it is a Cox ring. Then, for any divisor $D$ in the span $\mathcal{S}=\sum\R_+D_i$, finite generation implies the existence of  a birational map $\vphi_D\colon X\dashto X_D$, where $\vphi_D$ is a composition of elementary surgery operations that can be fully understood. Both $X_D$ and $(\vphi_D)_* D$ have good properties: $X_D$ is projective and $\Q$-factorial, and $(\vphi_D)_*D$ is semiample. 

In addition, there is a decomposition of $\mathcal{S}= \bigcup \mathcal{S}_j$ into finitely many rational polyhedral cones, together with birational maps $\vphi_j\colon X \dashto X_j$, such that the pushforward under $\vphi_j$ of every divisor in $\mathcal{S}_j$ is a nef divisor on $X_j$. In this paper, we say that these models $\vphi_j\colon X \dashto X_j$ are \emph{optimal}; the precise definition is in Section \ref{sec:auxiliary}. By analogy with the classical case, the map $\vphi_j\colon X \dashto X_j$ is called a \emph{$D$-MMP}. After Shokurov, the decomposition of $\mathcal S$ above is called a \emph{geography} of optimal models associated to $\mathfrak R$. 

The goal of this paper is twofold. On the one hand, we want to put these two families of varieties under the same roof. That is to say, we want to identify the maximal class of varieties and divisors on them where a suitable MMP can be performed. On the other hand, we want to understand why this class is the right one, i.e.\ what the key ingredients that make the MMP work are.

Let $D$ be a $\Q$-divisor on a variety $X$ in one of the two families above. The $D$-MMP has two significant  features, which we would like to extend to a more general setting:
\begin{enumerate}
\item[(i)] all varieties in the MMP are $\Q$-factorial,
\item[(ii)] the section ring $R(X,D)$ is preserved under the operations of the MMP.
\end{enumerate}
Condition (ii) is by now well understood: contracting maps that preserve sections of $D$ are $D$-nonpositive -- we recall this definition in Section~\ref{sec:auxiliary}. Somewhat surprisingly, preserving $\Q$-factoriality is the main obstacle to extending the MMP to arbitrary varieties $X$ and divisors $D$, even when the rings $R(X,D)$ are finitely generated; this is explained in  Section \ref{sec:amplemodels}. 

In this work, we introduce the notion of \emph{gen} divisors. We say that a $\Q$-divisor $D$ on a $\Q$-factorial projective variety $X$ is gen if every $\Q$-divisor in its numerical equivalence class has a finitely generated section ring. Ample divisors are examples of gen divisors. As we explain in Section \ref{sec:amplemodels}, in the situations of interest to us, these form essentially the only source of examples: indeed, all gen divisors there come from ample divisors on the end products of some MMP. However, one should bear in mind that semiample divisors are not necessarily gen. 

The main result of this paper is the following.

\begin{thm}\label{thm:main}
Let $X$ be a projective $\Q$-factorial variety, let $D_1,\dots,D_r$ be effective $\Q$-divisors on $X$, and assume that the numerical classes of $D_i$ span $N^1(X)_\R$. Assume that the ring $R(X;D_1,\dots,D_r)$ is finitely generated, that the cone $\sum\R_+D_i$ contains an ample divisor, and that every divisor in the interior of this cone is gen. 

Then there is a finite decomposition
$$\sum\R_+D_i= \coprod \mcal{N}_i$$
into cones having the following properties:
\begin{enumerate}
\item each $\overline{\mathcal N_i}$ is a rational polyhedral cone,
\item for each $i$, there exists a $\Q$-factorial projective variety $X_i$ and a birational contraction $\vphi_i\colon X\dashto X_i$ such that $\varphi_i$ is an optimal  model for every divisor in $\mathcal N_i$.
\end{enumerate}
\end{thm}

In fact, we prove a stronger result: we show that for any $\Q$-divisor $D\in\sum\R_+ D_i$, we can run a $D$-MMP which terminates, see Theorem~\ref{thm:scalingbig} for the precise statement. The decomposition in Theorem~\ref{thm:main} determines a geography of optimal models associated to $R(X;D_1,\dots,D_r)$. The techniques used in this paper build on and extend those from \cite{CoL10}.

Our  work has been  influenced by several lines of research. The original idea that geographies of various models are the right thing to look at is due to Shokurov \cite{Sho96}, and the first unconditional results were proved in \cite{BCHM}. Similar decompositions were considered in the context of Mori Dream Spaces by Hu and Keel \cite{HK00}, and as we demonstrate here, these are closely related to the study of asymptotic valuations in \cite{ELMNP}. Theorem \ref{thm:main} reproves and generalises some of the main results from these papers. We obtain in Corollary \ref{cor:lt models} the finiteness of models due to \cite{BCHM} by using the main theorem from \cite{CaL10}. Further, in Corollary \ref{cor:MDS} we prove a characterisation of Mori Dream Spaces in terms of the finite generation of their Cox rings due to \cite{HK00} without using GIT techniques. 

Along the way, we establish several results  of independent interest. Theorem~\ref{thm:CY}  shows that the nef and movable cones on Calabi-Yau varieties are locally rational polyhedral inside the big cone, and Lemma \ref{lem:equal proj} gives a sufficient condition for the stable base loci of numerically equivalent divisors to coincide. 

We spend a few words on the organisation of the paper. Section~\ref{sec:auxiliary} sets the notation and gathers some preliminary results. Section \ref{sec:fingen} focuses on finite generation of divisorial rings, and contains a number of results which are either of independent interest, or are used later in the paper.

In Section \ref{sec:amplemodels}, we show the existence of a decomposition $\sum\R_+D_i= \coprod \mcal{A}_i$ similar to that from Theorem~\ref{thm:main}, where all divisors in a given  chamber $\mcal A_i$ have a common \emph{ample model}, see Theorem \ref{thm:decomposition}. We study the geography of ample models. The main drawback of this decomposition is that the corresponding models are not $\Q$-factorial in general. Moreover, we show in Example \ref{exa:notgen} that the conditions of Theorem \ref{thm:decomposition} are not sufficient to ensure the existence of optimal models as in Theorem \ref{thm:main}. We explain why the presence of gen divisors is essential to the proof of Theorem \ref{thm:main}. However, we give a short proof that some of these models are indeed $\Q$-factorial in the case of adjoint divisors in Theorem \ref{lem:adjoint}.

In Section \ref{sec:DMMP}, we define what is meant by the MMP in our setting; it is easy to see that this generalises the classical MMP constructions. We then
prove Theorem \ref{thm:scalingbig}, which is a strengthening of Theorem \ref{thm:main}. The main technical result is Theorem \ref{lem:nullflip}, and the
presence of gen divisors is essential to its proof. We mention here that this reveals the philosophical role of the gen condition: it enables one to prove a
version of the classical Basepoint free theorem, which is why we can then run the Minimal Model Program and preserve $\Q$-factoriality in the process. We end
the paper with several corollaries that recover quickly some of the main results from \cite{BCHM} and \cite{HK00}.
\section{Preliminary results}\label{sec:auxiliary}

Throughout this paper we work with varieties defined over
$\mathbb C$. Unless otherwise stated, all varieties are projective and normal. We denote by $\mathbb R_+$ and $\mathbb Q_+$ the sets of non-negative real and rational numbers. 

\subsection*{Convex geometry} 

Let $\mathcal{C}\subseteq \mathbb R^N$ be a convex set. A subset $F\subseteq\mathcal C$ is a \emph{face} of $\mathcal{C}$ if it is convex, and whenever
$tu+(1-t)v\in F$ for some $u,v\in\mathcal C$ and $0<t<1$, then $u,v\in F$.  

The topological closure of a set $\mathcal S\subseteq\R^N$ is denoted by $\overline{\mcal S}$. The boundary of a closed set $\mcal C\subseteq\R^N$ is denoted by $\partial\mcal C$.

A \emph{rational polytope} in $\mathbb R^N$ is a compact set which is the convex hull of finitely many rational points in $\mathbb R^N$. A \emph{rational polyhedral cone} in $\mathbb R^N$ is a convex cone spanned by finitely many rational vectors. The dimension of a cone in $\R^N$ is the dimension of the minimal $\R$-vector space containing it.

A finite rational polyhedral subdivision $\mathcal C=\bigcup \mcal C_i$ of a rational polyhedral cone $\mcal C$ is a \emph{fan} if each face of $\mcal C_i$ is also a cone in the decomposition, and the intersection of two cones in the decomposition is a face of each. 

\subsection*{Divisors and line bundles}

Let $X$ be a normal projective variety and let $\mathbf k\in \{\mathbb Z,\mathbb Q,\mathbb R\}$. We denote by $\Div_{\mathbf k}(X)$ the group of $\mathbf k$-Cartier $\mathbf k$-divisors on $X$, and $\sim_{\mathbf k}$ and $\equiv$ denote the $\mathbf k$-linear and numerical equivalence of $\mathbb R$-divisors. If there is a morphism $X\lto Y$ to another normal projective variety, numerical equivalence over $Y$ is denoted by $\equiv_Y$. We denote $\Pic(X)_\mathbf k=\Div_\mathbf k(X)/\sim_\mathbf k$ and $N^1(X)_\mathbf k=\Div_\mathbf k(X)/\equiv$. 

The ample, big, nef, effective, and pseudo-effective cones in $N^1(X)_\R$ are denoted by $\Amp (X)$, $\B(X)$, $\Nef (X)$, $\Eff(X)$, and $\overline{\Eff}(X)$. The movable cone $\overline{\Mov}(X)$ is the closure of the cone in $N^1(X)_\R$ spanned by the classes of divisors whose base locus has codimension at least $2$.

If $X$ is a normal projective variety, and if $D$ is an $\R$-divisor on $X$, then the group of global sections of $D$ is 
\[
H^0(X,D)=\bigl\{f\in k(X)\mid \ddiv f+D \geq 0  \bigr\},
\]
and the associated \emph{section ring} is defined as
$$R(X,D)=\bigoplus_{m\in\N}H^0(X,mD).$$

A {\em pair} $(X,\Delta)$ consists of a normal projective variety $X$ and an $\R$-divisor $\Delta\geq0$ on $X$ such that $K_X+\Delta$ is $\R$-Cartier. When $(X,\Delta)$ is a pair, $K_X+\Delta$ is an {\em adjoint divisor}. The pair $(X,\Delta)$ has {\em klt\/} (respectively \emph{log canonical}) singularities if for every log resolution $f\colon Y\longrightarrow X$, the divisor $K_Y-f^*(K_X+\Delta)$ has all coefficients $>{-}1$ (respectively $\geq{-}1$).

A projective variety $X$ is said to be of \emph{Calabi-Yau type} if there exists a $\Q$-divisor $\Delta\geq0$ such that $(X,\Delta)$ is klt and $K_X+\Delta\equiv0$.

If $X$ is a normal projective variety, and if $D$ is an integral divisor on $X$, we denote by $\Bs|D|$ the base locus of $D$, whereas $\Fix|D|$ and $\Mob(D)$ denote the fixed and mobile parts of $D$. If $D$ is an $\mathbb R$-divisor on $X$, we denote
$$|D|_{\R}=\{D'\geq0\mid D\sim_{\R} D'\}\qquad\text{and}\qquad \sB(D)=\bigcap_{D'\in|D|_\mathbb R}\Supp D',$$
and we call $\sB(D)$ the {\em stable base locus} of $D$. We set $\sB(D)=X$ if $|D|_\mathbb R=\emptyset$. 

As mentioned in the introduction, we pay special attention to two classes of varieties: varieties for which the classical MMP can be performed successfully on the one hand, and Mori Dream Spaces on the other. We recall some definitions related to their study. 

\begin{dfn}\label{d_variouspolytopes}
Let $X$ be a projective $\Q$-factorial variety, let $S_1,\dots,S_p$ be distinct prime divisors on $X$, denote $V=\sum_{i=1}^p\mathbb R S_i\subseteq \Div_{\mathbb R}(X)$, and let $A$ be an ample $\mathbb Q$-divisor on $X$. We define
\begin{align*}
\mathcal L(V)&=\{ \Delta\in V \mid (X,\Delta)\text{ is log canonical}\},\\
\mathcal E_A(V)&=\{\Delta\in\mathcal L(V)\mid |K_X+A+\Delta|_{\R}\neq\emptyset\}.
\end{align*}
\end{dfn}
It is easy to check that $\mathcal L(V)$ is a rational polytope, cf.\ \cite[Lemma 3.7.2]{BCHM}. On the other hand, the fact that $\mathcal E_A(V)$ is a rational polytope is much harder, see Corollary \ref{cor:effective}.

\begin{dfn}\label{dfn:25}
A projective $\Q$-factorial variety $X$ is a Mori Dream Space if
\begin{enumerate}
\item $\Pic(X)_\Q =N^1(X)_\Q$,
\item $\Nef(X)$ is the affine hull of finitely many semiample line bundles, and
\item there are finitely many birational maps $f_i\colon X \dashrightarrow X_i$ to projective $\Q$-factorial varieties $X_i$ such that each $f_i$ is an isomorphism in codimension $1$, each $X_i$ satisfies (2), and $\overline{\Mov}(X)=\bigcup f^*_i\big(\Nef(X_i)\big)$.
\end{enumerate}
\end{dfn}

\subsection*{Models}

Let $X$ be a normal projective variety and let $D$ be an $\R$-Cartier divisor on $X$. We adopt some of the definitions of \emph{models} of $D$ from \cite{BCHM}. The models defined below are significant  because they provide  correct generalisations of minimal and canonical models for divisors which are not necessarily adjoint.

\begin{dfn}\label{dfn1}
Let $D\in \Div_{\R}(X)$ and let $\vphi \colon X \dashto Y$ be a contraction map to a normal projective variety $Y$ such that $D'= \vphi_{*}D$ is $\R$-Cartier.
\begin{enumerate}
\item The map $\vphi$ is \emph{$D$-nonpositive} (respectively \emph{$D$-negative}) if it is birational, and for a common resolution $(p,q)\colon W \longrightarrow X\times Y$, we can write
$$p^*D = q^*D'+ E,$$
where $E\geq 0$ is $q$-exceptional (respectively $E\geq 0$ is $q$-exceptional and $\Supp E$ contains the strict transform of the $\vphi$-exceptional divisors).
\item The map $\vphi$ is an \emph{optimal model} of $D$ if $\vphi$ is $D$-negative, $Y$ is $\Q$-factorial and $D'$ is nef. 
\item The map $\vphi$ is a \emph{semiample model} of $D$ if $\vphi$ is $D$-nonpositive and $D'$ is semiample. 
\item The map $\vphi$ is the \emph{ample model} of $D$ if there exist a birational contraction $f\colon X\dashrightarrow Z$ which is a semiample model of $D$, and a morphism with connected fibres $g\colon Z\longrightarrow Y$ such that $\varphi=g\circ f$ and $f_*D=g^*A$, where $A$ is an ample $\R$-divisor on $Y$. 
\end{enumerate}
\end{dfn}

\begin{rem}\label{rem:1}
\begin{enumerate}
\item[(i)] If the birational contraction $\vphi \colon X \dashto Y$ is $D$-nonpositive for some $D\in \Div_{\Q}(X)$, then $H^0(X,D)\simeq H^0(Y,\varphi_*D)$.
\item[(ii)] The ample model is unique up to isomorphism. Indeed, with the notation from Definition \ref{dfn1}, we may assume that $D$ is a $\Q$-divisor, and (i) shows that $R(X,D)\simeq R(Z,f_*D)$. This last ring is isomorphic to $R(Y,A)$, and therefore $Y\simeq\Proj R(X,D)$. Note that in Definition \ref{dfn1}(4), when $D$ is a $\Q$-divisor, the map $g$ is the \emph{semiample fibration} associated to $D$, see \cite[Theorem 2.1.27]{Laz04}.
\item[(iii)] In this paper we require that the ample model factors through a semiample model (compare with \cite[Definition 3.6.5, Lemma 3.6.6(3)]{BCHM}).
\item[(iv)] When $D$ is an adjoint divisor, then optimal, semiample and ample models are called \emph{log terminal}, \emph{good} and \emph{log canonical} models, respectively.
\end{enumerate}
\end{rem}

We recall the following important result known as the Negativity lemma, see \cite[Lemma 2.19]{Kol92}. This result and its corollary will be used in
Sections~\ref{sec:amplemodels} and \ref{sec:DMMP}.

\begin{lem}\label{lem:negativity}
Let $f\colon X \lto Y$ be a proper birational morphism, where $X$ is normal, and let $E$ be an f-exceptional divisor on $X$. Assume that $E\equiv_Y H + D$, where $H$ is $f$-nef and $D\geq0$ has no common components with $E$. 

Then $E\leq 0$.
\end{lem}

\begin{cor}\label{cor:negative}
Let $X\lto Z$ and $Y \lto Z$ be projective morphisms of normal projective varieties. Let $f\colon X \dashto Y$ be a birational contraction over $Z$, and let $(p,q)\colon W\lto X\times Y$ be a resolution of $f$. Let $D$ and $D'$ be $\R$-Cartier divisors on $X$ such that $f_*D$ and $f_*D'$ are $\R$-Cartier on $Y$, and assume that $D\equiv_Z D'$. Then
$$p^*D-q^*f_*D = p^*D'-q^*f_*D'.$$
In particular, $f$ is $D$-nonpositive (respectively $D$-negative) if and only if $f$ is $D'$-nonpositive (respectively $D'$-negative).
\end{cor}
\begin{proof}
This result is \cite[Lemma 3.6.4]{BCHM}. The divisor $E=p^*(D-D')-q^*f_*(D-D')$ is $q$-exceptional since $f$ is a contraction, and we have $E\equiv_Y0$. We conclude by Lemma \ref{lem:negativity}.
\end{proof}

\subsection*{Asymptotic valuations}
A {\em geometric valuation\/} $\Gamma$ over a normal variety $X$ is a valuation on $k(X)$ given by the order of vanishing at the generic point of a prime divisor on some birational model $f\colon Y\longrightarrow X$. If $D$ is an $\R$-Cartier divisor on $X$, we abuse notation and write $\mult_\Gamma D$ to denote $\mult_\Gamma f^*D$.

The following definition is due to Nakayama.

\begin{dfn}\label{dfn:1}
Let $X$ be a normal projective variety, let $D$ be an $\R$-Cartier divisor such that $|D|_\R\neq\emptyset$, and let $\Gamma$ be a geometric valuation over $X$. The {\em asymptotic order of vanishing\/} of $D$ along $\Gamma$ is 
$$o_\Gamma (D)=\inf\{\mult_\Gamma D'\mid D'\in|D|_\R\}.$$
If $D$ is a big divisor, we define
\[
\textstyle N_\sigma(D)=\sum_\Gamma o_\Gamma(D)\cdot\Gamma\quad\text{ and }\quad  P_\sigma(D) = D - N_\sigma(D),
\]
where the sum runs over all prime divisors $\Gamma$ on $X$.
\end{dfn}

\begin{rem}
On a surface $X$, the construction above gives the classical Zariski decomposition: this is a unique decomposition $D=P_\sigma(D)+N_\sigma(D)$, where $P_\sigma(D)$ is nef, and $N_\sigma(D)=\sum\gamma_i \Gamma_i$ is an effective divisor such that $P_\sigma(D)\cdot\Gamma_i=0$ for all $i$, and the matrix $(\Gamma_i\cdot\Gamma_j)$ is negative definite. We use this characterisation in Example \ref{exa:notgen}.
\end{rem}

\begin{lem}\label{d_sigma}
Let $X$ be a $\Q$-factorial projective variety, let $D$ be a big $\mathbb R$-divisor, and let $\Gamma$ be a prime divisor. Then $o_\Gamma(D)$ depends only on the numerical class of $D$. The function $o_\Gamma$ is homogeneous of degree one, convex and continuous on $\B(X)$. The formal sum $N_\sigma(D)$ is an $\mathbb R$-divisor, the divisor $P_\sigma(D)$ is movable, and for any $\mathbb R$-divisor $0\le F\le N_\sigma(D)$ we have $N_\sigma(D-F)=N_\sigma(D)-F$. If $E\geq0$ is an $\R$-divisor on $X$ such that $D-E\in\overline{\Mov}(X)$, then $E\geq N_\sigma(D)$.
\end{lem}
\begin{proof}
See \cite[\S III.1]{Nak04}.
\end{proof}

In certain situations we have more information on the divisor $P_\sigma(D)$.

\begin{lem}\label{lem:projectionMov}
Let $X$ be a $\Q$-factorial projective variety, and let $D$ be a big $\Q$-divisor on $X$. Assume that the cone $\overline{\Mov}(X)$ is rational polyhedral. 

Then $P_\sigma(D)$ is a $\Q$-divisor, and $R(X,D)$ is finitely generated if and only if $R(X,P_\sigma(D))$ is finitely generated.
\end{lem}
\begin{proof}
Let $\Gamma_i$ be the components of $N_\sigma(D)$, and denote 
$$\mathcal H=D-\sum\R_+\Gamma_i\quad\text{and}\quad\mathcal G=P_\sigma(D)-\sum\R_+\Gamma_i.$$ 
Then we have $\overline{\Mov}(X)\cap\mathcal H\subseteq\mathcal G$ by Lemma \ref{d_sigma}. Since $\overline{\Mov}(X)\cap\mathcal H$ is an intersection of finitely many rational half-spaces, and as $P_\sigma(D)\in\overline{\Mov}(X)$ is an extremal point of $\mathcal G$, we conclude that $P_\sigma(D)$ is a $\Q$-divisor. 

For the second statement, we may assume that $D$ is an integral divisor and that $|D|\neq\emptyset$, so the claim follows from $P_\sigma(mD)\geq\Mob(mD)$ for every positive integer $m$.
\end{proof}

The proof of the following lemma is analogous to that of \cite[Lemma 5.2]{CoL10}, and it will be used in Section \ref{sec:DMMP} to ensure that a certain MMP terminates.

\begin{lem}\label{lem:2}
Let $f\colon X\dashrightarrow Y$ be a birational contraction between projective $\Q$-factorial varieties, and let $\mathcal C\subseteq\Div_\R(X)$ be a cone such that $f$ is $D$-nonpositive for all $D\in\mathcal C$. Let $\Gamma$ be a geometric valuation on $k(X)$.

Then $o_\Gamma$ is linear on $\mathcal C$ if and only if it is linear on the cone $f_*\mathcal C\subseteq\Div_\R(Y)$.
\end{lem}
\begin{proof}
Let $(p,q)\colon W\longrightarrow X\times Y$ be a resolution of $f$. Then for every $D\in\mathcal C$ we have $p^*D=q^*f_*D+E_D$, where
$E_D\geq0$ is a $q$-exceptional divisor. This implies that $f_*$ restricts to an isomorphism between $|D|_\R$ and $|f_*D|_\R$. Denote 
$$V_D=\{D_X-D\mid D_X\in|D|_\R\}\quad\text{and}\quad W_D=\{D_Y-f_*D\mid D_Y\in|f_*D|_\R\}.$$
By the above, we have the isomorphism $f_*|_{V_D}\colon V_D\simeq W_D$, and also $\mult_\Gamma P_X=\mult_\Gamma f_*P_X$ for every $P_X\in V_D$ by \cite[Lemma 5.1(2)]{CoL10}. Therefore
\[
o_\Gamma(D)-\mult_\Gamma D=\inf_{P_X\in V_D}\mult_\Gamma
P_X=\inf_{P_X\in V_D}\mult_\Gamma f_*P_X=o_\Gamma(f_*D)-\mult_\Gamma
f_*D,
\]
hence the function $o_\Gamma(\cdot)-o_\Gamma\big(f_*(\cdot)\big)\colon \mcal C\longrightarrow\R$ is equal to the linear map $\mult_\Gamma(\cdot)-\mult_\Gamma f_*(\cdot)$.  The lemma follows.
\end{proof}
\section{Around finite generation}\label{sec:fingen}

In this section, we recall the definition of divisorial rings on a normal projective variety from \cite{CaL10,CoL10}, and we give some examples that we use later -- adjoint rings, and rings spanned by big divisors on varieties of Calabi-Yau type. We then relate finite generation to asymptotic valuations and to properties of the stable base locus.

\subsection*{Divisorial rings}
If $X$ is a normal projective variety, and if $\mathcal S\subseteq\Div_{\mathbb Q}(X)$ is a finitely generated monoid, then
$$R(X,\mathcal{S})=\bigoplus_{D\in\mathcal{S}}H^0\big(X,D)$$
is a {\em divisorial $\mathcal S$-graded ring\/}. If $\mathcal C\subseteq\Div_{\mathbb R}(X)$ is a rational polyhedral cone, then $\mathcal S=\mathcal C\cap\Div(X)$ is a finitely generated monoid by Gordan's lemma, and we define the ring $R(X,\mathcal C)$ to be $R(X,\mathcal S)$. We also use divisorial rings of the form
\[
\mathfrak R=R(X;D_1, \dots, D_r)=\bigoplus_{(n_1,\dots, n_r)\in \N^r} H^0(X,
n_1D_1+\dots + n_rD_r),
\]
where $D_1,\dots, D_r\in\Div_\Q(X)$. If $D_i$ are adjoint divisors, the ring $\mathfrak R$ is an {\em adjoint ring\/}. The {\em support\/} of $\mathfrak R$ is the cone
\[
\textstyle\Supp \mathfrak R=\{D\in\sum\R_+D_i\mid |D|_\R\neq\emptyset\}\subseteq \Div_\R(X),
\]
and similarly for rings of the form $R(X,\mathcal C)$.

If $X$ is a $\Q$-factorial projective with $\Pic(X)_\Q= N^1(X)_\Q$, and if $D_1, \dots, D_r$ is a basis of $\Pic(X)_\Q$ such that $\Effb(X)\subseteq \sum \R_+D_i$, then $R(X;D_1,\dots,D_r)$ is a \emph{Cox ring} of $X$. The finite generation of this ring is independent of the choice of $D_1,\dots,D_r$. 

Throughout the paper, we use several properties of finitely generated divisorial rings without explicit mention, see \cite[\S 2.4]{CaL10} for details and background. The one we use most is recalled in the following lemma. 

\begin{lem}\label{lem:3}
Let $X$ be a normal projective variety, let $D_1,\dots,D_r$ be divisors in $\Div_\Q(X)$, and let $p_1,\dots,p_r$ be positive rational numbers. 

Then $R(X;D_1,\dots,D_r)$ is finitely generated if and only if $R(X;p_1D_1,\dots,p_rD_r)$ is finitely generated. 
\end{lem}

\subsection*{Relation to asymptotic valuations}
Finite generation of a divisorial ring $\mathfrak R$ has important consequences on the behavior of the asymptotic order functions, and therefore on the convex geometry of $\Supp \mathfrak R$, as observed in \cite{ELMNP}.   

\begin{thm}
\label{thm:ELMNP}
Let $X$ be a projective $\Q$-factorial variety, and let $\mathcal C\subseteq\Div_\R(X)$ be a rational polyhedral cone. Assume that the ring $\mathfrak R=R(X,\mcal C)$ is finitely generated. Then:
\begin{enumerate}
\item $\Supp \mathfrak R$ is a rational polyhedral cone,
\item if $\Supp \mathfrak R$ contains a big divisor, then all pseudo-effective divisors in $\Supp\mathfrak R$ are in fact effective,
\item there is a finite rational polyhedral subdivision $\Supp \mathfrak R=\bigcup \mcal{C}_i$ such that $o_\Gamma$ is linear on $\mcal{C}_i$ for every geometric valuation $\Gamma$ over $X$, and the cones $\mcal C_i$ form a fan,
\item there is a positive integer $d$ and a resolution $f \colon \widetilde{X} \lto X$ such that $\Mob f^*(dD)$ is basepoint free for every $D\in \Supp \mathfrak R\cap \Div(X)$, and $\Mob f^*(kdD)=k\Mob f^*(dD)$ for every positive integer $k$.
\end{enumerate}
\end{thm}
\begin{proof} 
This is essentially \cite[Theorem 4.1]{ELMNP}, see \cite[Theorem 3.6]{CoL10}.
\end{proof}

Part (i) of the following lemma is \cite[Lemma 3.8]{CoL10}. Part (ii) is a result of Zariski and Wilson, cf.\ \cite[Theorem 2.3.15]{Laz04}.

\begin{lem}\label{lem:ords}
Let $X$ be a normal projective variety and let $D$ be a divisor in $\Div_\Q(X)$.
\begin{enumerate}
\item[(i)] If $|D|_\Q\neq\emptyset$, then $D$ is semiample if and only if $R(X, D)$ is finitely generated and $o_{\Gamma}(D)=0$ for all geometric valuations $\Gamma$ over $X$.
\item[(ii)] If $D$ is nef and big, then $D$ is semiample if and only if $R(X, D)$ is finitely generated.
\end{enumerate}
\end{lem}
\begin{proof} 
If $D$ is semiample, then some multiple of $D$ is basepoint free, thus $R(X,D)$ is finitely generated by Lemma \ref{lem:3}, and all $o_\Gamma(D)=0$. Now, fix a point $x\in X$. If $R(X,D)$ is finitely generated and $o_x(D)=0$, then $x\notin\sB(D)$ by Theorem \ref{thm:ELMNP}(4), which proves (i).

For (ii), let $A$ be an ample divisor. Then $D+\varepsilon A$ is ample for any $\varepsilon>0$, hence $o_\Gamma(D+\varepsilon A)=0$ for any geometric valuation $\Gamma$ over $X$. But then $o_\Gamma(D)=\lim\limits_{\varepsilon\to0}o_\Gamma(D+\varepsilon A)=0$ by Lemma \ref{d_sigma}, so we conclude by (i).
\end{proof}

\begin{cor}\label{cor:5}
Let $X$ be a normal projective variety and let $D_1, \dots, D_r$ be divisors in $\Div_\Q(X)$. Assume that the ring $\mathfrak R=R(X;D_1, \dots, D_r)$ is finitely generated, and let $\Supp \mathfrak R=\bigcup_{i=1}^N \mcal{C}_i$ be a finite rational polyhedral subdivision as in Theorem~\ref{thm:ELMNP}(3). Denote by $\pi\colon\Div_\R(X)\lto N^1(X)_\R$ the natural projection.

Then there is a set $I_1\subseteq\{1,\dots,N\}$ such that $\Supp\mathfrak R\cap \pi^{-1}\bigl(\overline{\Mov} (X)\bigr)=\bigcup_{i\in I_1}\mathcal C_i$.

Assume further that $\Supp\mathfrak R$ contains an ample divisor. Then there is a set $I_2\subseteq\{1,\dots,N\}$ such that the cone $\Supp\mathfrak R\cap \pi^{-1}\bigl(\Nef(X)\bigr)$ equals $\bigcup_{i\in I_2}\mathcal C_i$, and every element of this cone is semiample.
\end{cor}
\begin{proof}
For every prime divisor $\Gamma$ on $X$ denote $\mathcal C_\Gamma=\{D\in\Supp\mathfrak R\mid o_\Gamma(D)=0\}$. If $\mcal C_\Gamma$ intersects the interior of some $\mcal C_\ell$, then $\mcal C_\ell\subseteq\mcal C_\Gamma$ since $o_\Gamma$ is a linear non-negative function on $\mcal C_\ell$. Therefore, there exists a set $I_\Gamma \subseteq\{1,\dots,N\}$ such that $\mathcal C_\Gamma=\bigcup_{i\in I_\Gamma}\mathcal C_i$. Now the first claim follows since $\overline{\Mov}(X)$ is the intersection of all $\mathcal C_\Gamma$.

For the second claim, note that since $\Supp\mathfrak R\cap \pi^{-1}\bigl(\Nef(X)\bigr)$ is a cone of dimension $\dim\Supp\mathfrak R$, we can consider only maximal dimensional cones $\mcal C_\ell$. Now, for every $\mcal C_\ell$ whose interior contains an ample divisor, all asymptotic order functions $o_\Gamma$ are identically zero on $\mcal{C}_\ell$ similarly as above. Therefore, by Lemma \ref{lem:ords}, every element of $\mcal{C}_\ell$ is semiample, and thus $\mcal{C}_\ell\subseteq \Supp\mathfrak{R}\cap\pi^{-1} \bigl(\Nef(X)\bigr)$. The claim follows.
\end{proof}

\subsection*{Examples of finitely generated rings}
The following is a small variation of the main result of \cite{CaL10}, where it is proved by a self-contained argument avoiding the techniques of the MMP. It was first proved in the seminal paper \cite{BCHM} by MMP methods.

\begin{thm}\label{thmA}
Let $X$ be a $\Q$-factorial projective variety, and let $\Delta_1,\dots,\Delta_r$ be $\Q$-divisors such that all pairs $(X,\Delta_i)$ are klt.
\begin{enumerate}
\item If $A_1,\dots,A_r$ are ample $\Q$-divisors, then the adjoint ring
\[
R(X; K_X+\Delta_1+A_1, \dots, K_X+\Delta_r+A_r)
\]
is finitely generated.
\item If $\Delta_i$ are big, then the adjoint ring
\[
R(X;K_X+\Delta_1,\dots,K_X+\Delta_r)
\]
is finitely generated.
\end{enumerate}
\end{thm}
\begin{proof} 
See \cite[Theorem 3.2]{CoL10}.
\end{proof}

\begin{cor}\label{cor:effective}
Let $X$ be a projective $\Q$-factorial variety, let $S_1,\dots,S_p$ be distinct prime divisors on $X$, denote $V=\sum_{i=1}^p\mathbb R S_i\subseteq \Div_{\mathbb R}(X)$, and let $A$ be an ample $\mathbb Q$-divisor on $X$. Let $\mcal C\subseteq\mcal L(V)$ be a rational polytope such that for every $\Delta\in \mcal C$, the pair $(X,\Delta)$ is klt.

Then the set $\mcal C\cap\mathcal E_A(V)$ is a rational polytope, and the ring $R(X,\R_+(K_X+A+\mcal C\cap\mathcal E_A(V)))$ is finitely generated.
\end{cor}
\begin{proof}
Let $B_1,\dots,B_r$ be the vertices of $\mcal C$. Then the ring $\mathfrak R=R(X;K_X+B_1+A,\dots,K_X+B_r+A)$ is finitely generated by Theorem \ref{thmA}, and we have $\Supp\mathfrak R=\R_+(K_X+A+\mcal C\cap\mcal E_A(V))$. Now the result follows from Theorem \ref{thm:ELMNP}(1).
\end{proof}

The following corollary is well known.

\begin{cor}\label{cor:CY finitely generated}
Let $X$ be a projective $\Q$-factorial variety of Calabi-Yau type, and let $B_1,\dots,B_q$ be big $\Q$-divisors on $X$. Then the ring $R(X;B_1,\dots,B_q)$ is finitely generated. 
\end{cor}
\begin{proof}
Let $\Delta\geq0$ be a $\Q$-divisor such that $(X,\Delta)$ is klt and $K_X+\Delta\equiv 0$, and write $B_i=A_i+E_i$, where each $A_i$ is ample and $E_i\geq0$. Let $\varepsilon>0$ be a rational number such that all pairs $(X,\Delta+\varepsilon E_i)$ are klt, and denote $A_i'=\varepsilon B_i-(K_X+\Delta+\varepsilon E_i)$. Then each $A_i'$ is ample since $A_i'\equiv\varepsilon A_i$, hence the adjoint ring
$$R(X;K_X+\Delta+\varepsilon E_1+A_1',\dots,K_X+\Delta+\varepsilon E_q+A_q')= R(X;\varepsilon B_1,\dots, \varepsilon B_q )$$ 
is finitely generated by Theorem \ref{thmA}. Therefore $R(X;B_1,\dots,B_q)$ is finitely generated by Lemma \ref{lem:3}.
\end{proof}

Part (1) of the following theorem was proved in \cite{Kaw88}, while part (2) is a generalisation of the analogous result for $3$-folds proved in \cite{Kaw97}.

\begin{thm}\label{thm:CY}
Let $X$ be a projective $\Q$-factorial variety of Calabi-Yau type.
\begin{enumerate}
\item The cone $\Nef(X)$ is locally rational polyhedral in $\B(X)$, and moreover, every element of $\Nef(X)\cap\B(X)$ is semiample.
\item The cone $\overline{\Mov}(X)$ is locally rational polyhedral in $\B(X)$.
\end{enumerate} 
\end{thm}
\begin{proof}
We prove the result for $\overline{\Mov}(X)$; the case of $\Nef(X)$ is analogous. 

Let $V$ be a relatively compact subset of the boundary of $\overline{\Mov}(X)\cap\B(X)$, and denote by $\pi\colon \Div_\R(X)\lto N^1(X)_\R$ the natural projection. Then we can choose finitely many big $\mathbb Q$-divisors $B_1,\dots,B_q$ such
that $V\subseteq\pi(\sum_{i=1}^q\mathbb R_+B_i)$. Corollary~\ref{cor:CY finitely generated} implies that the ring  $\mathfrak R=R(X;B_1,\dots,B_q)$ is finitely generated, and hence $\pi^{-1}\big(\overline{\Mov}(X)\big)\cap \Supp\mathfrak R$ is a rational polyhedral cone by Corollary \ref{cor:5}. But then $V$ is contained in finitely many rational hyperplanes. 
\end{proof}

\begin{rem}
The proof of Theorem \ref{thm:CY} shows existence of a locally polyhedral decomposition of the big cone on a variety $X$ of Calabi-Yau type, which comes from the behaviour of asymptotic valuations on this cone. This is a consequence of the finite generation of any divisorial ring which is supported on $\B(X)$. On the other hand, it was shown in \cite{BKS} that an analogous decomposition exists on any smooth surface, and it is not a consequence of finite generation. It is an interesting problem to establish in which contexts that result generalises to higher dimensions. 
\end{rem}

\subsection*{Relations to the stable base locus}

As Example \ref{exa1} shows, the stable base locus and finite generation of section rings are not, in general, numerical invariants. However, we prove in Lemma \ref{lem:equal proj} that under some finite generation hypotheses, the stable base loci of numerically equivalent big divisors coincide. 

\begin{exa}\label{exa1}
We recall \cite[Example 10.3.3]{Laz04}. Let $B$ be a smooth elliptic curve, and let $A$ be an ample divisor of degree $1$ on $B$. Let $X=\PS(\OO_B\oplus \OO_B(A))$ be a projective bundle with the natural map $p \colon X\lto B$. Let $P_1$ be a torsion divisor on $B$, let $P_2$ be a non-torsion degree $0$ divisor on $B$, and consider $L_i=  \OO_X(1)\otimes p^*\OO_{B}(P_i)$. Then $L_1$ and $L_2$ are numerically equivalent nef and big line bundles
with $\emptyset=\sB(L_1)\neq \sB(L_2)$, and $R(X, L_1)$ is finitely generated while $R(X, L_2)$ is not by Lemma \ref{lem:ords}(2).
\end{exa}

\begin{lem}\label{lem:equal proj}
Let $X$ be a $\Q$-factorial projective variety, and let $D_1$ and $D_2$ be big $\Q$-divisors such that $D_1\equiv D_2$. Assume that the rings $R(X,D_i)$ are finitely generated, and consider the maps $\vphi_i\colon X\dashto \Proj R(X,D_i)$.

Then we have $\sB(D_1)=\sB(D_2)$, and there is an isomorphism $\eta\colon \Proj R(X,D_1)\lto \Proj R(X,D_2)$ such that $\vphi_2=\eta\circ\vphi_1$.
\end{lem}
\begin{proof}
Since finite generation holds, we have $\sB(D_i)=\{x\in X\mid o_x(D_i)>0\}$, so the first claim follows immediately from Lemma \ref{d_sigma}.

For the second claim, by passing to a resolution and by Theorem \ref{thm:ELMNP}, we may assume that there is a positive integer $k$ such that $\Mob(kD_i)$ are basepoint free, and $\Mob(pkD_i)=p\Mob(kD_i)$ for all positive integers $p$. Note that then $P_\sigma(D_i)=\frac1k\Mob(kD_i)$, and that 
\begin{align}\label{equ:psigma}
P_\sigma(D_1)\equiv P_\sigma(D_2)
\end{align}
since $N_\sigma(D_1)=N_\sigma(D_2)$ by Lemma \ref{d_sigma}. Thus $\varphi_i$ is given by the linear system $|kp P_\sigma(D_i)|$ for some $p\gg0$. But then \eqref{equ:psigma} shows that $\vphi_1$ and $\vphi_2$ contract the same curves, which implies the claim.
\end{proof}



\section{Geography of ample models}\label{sec:amplemodels}

In this section we study the geography of ample models associated to a finitely generated divisorial ring $\mathfrak R= R(X; D_1, \dots , D_r)$. More precisely, there is a decomposition $\Supp \mathfrak{R}= \coprod \mathcal A_i$ into finitely many chambers together with contracting maps $\vphi_i\colon X\dashto X_i$, such that $\vphi_i$ is the ample model for every divisor in $\mcal A_i$. We study these ample models in the special case of adjoint divisors; then, the varieties $X_i$ are $\Q$-factorial when the numerical classes of the elements of $\mcal A_i$ span $N^1(X)_\R$. This is a highly desirable feature which we would like to preserve in the general case. We then formally introduce the gen condition, and show -- both by analysis and by example -- that it is necessary in order to perform a Minimal Model Program in a more general setting.

We first recall the following important result \cite[Proposition 1.2]{Rei80}. We follow closely the proof of \cite[Lemma 7.10]{Deb01}.
  
\begin{lem}\label{l:reid}
Let $X$ be a smooth variety and let $D$ be a big divisor on $X$.
Assume that, for every positive integer $m$, the divisor $M_m=\Mob(mD)$ is basepoint free, that $M_m=mM_1$, and that $\Fix|D|$ has simple normal crossings. Let $\varphi\colon X\longrightarrow Y$ be the semiample fibration associated to $M_1$. 

Then every component of $\Fix|D|$ is contracted by $\varphi$. In particular, we have $R(X,D)\simeq R(Y,\vphi_*D)$.
\end{lem}
\begin{proof}
Denote $n=\dim X$. We may assume that $\varphi$ is the morphism associated to $M_1$, and then $\OO_X(M_1)=\varphi^*\OO_Y(1)$ for a very ample line bundle $\OO_Y(1)$ on $Y$. Let $\Gamma$ be a component of $\Fix|D|$. We need to show that $h^0(\vphi(\Gamma),\OO_{\vphi(\Gamma)}(m))\leq O(m^{n-2})$. 

Since $\OO_X(M_m)=\vphi^*\OO_Y(m)$ and the natural map $\OO_{\vphi(\Gamma)}\lto\vphi_*\OO_\Gamma$ is injective, we have
\begin{equation}\label{equ:compare}
h^0(\vphi(\Gamma),\OO_{\vphi(\Gamma)}(m))\leq h^0(\vphi(\Gamma),\OO_Y(m)\otimes\vphi_*\OO_\Gamma)=h^0(\Gamma,\OO_\Gamma(M_m)).
\end{equation}
Write $\Gamma_{|\Gamma}\sim G^+-G^-$, where $G^+,G^-\geq0$ are Cartier divisors on $\Gamma$. Consider the exact sequences
\begin{equation}\label{equ:exseq2}
0\lto H^0(\Gamma, M_{m|\Gamma}-G^-)\longrightarrow H^0(\Gamma,M_{m|\Gamma}) \longrightarrow H^0(G^-, M_{m|G^-}) 
\end{equation}
and 
\begin{equation}\label{equ:exseq}
H^0(X, M_m)\longrightarrow H^0(X,M_m+\Gamma) \longrightarrow H^0(\Gamma, (M_m+\Gamma)_{\vert \Gamma}) \longrightarrow H^{1}(X, M_m).
\end{equation}
Since $\Fix|mD|=m\Fix|D|$, the divisor $\Gamma$ is a component of $\Fix|mD|$, hence the first map in \eqref{equ:exseq} is an isomorphism and the last map in \eqref{equ:exseq} is an injection. Therefore, from \eqref{equ:compare}, \eqref{equ:exseq2} and \eqref{equ:exseq} we have
\begin{multline*}
h^0(\vphi(\Gamma),\OO_{\vphi(\Gamma)}(m))\leq h^0(\Gamma,M_{m|\Gamma})\leq h^0(\Gamma, M_{m|\Gamma}-G^-)+h^0(G^-, M_{m|G^-})\\
\leq h^0(\Gamma, (M_m+\Gamma)_{|\Gamma})+h^0(G^-, M_{m|G^-})\leq h^1(X, M_m)+h^0(G^-, M_{m|G^-}).
\end{multline*}
As $h^0(G^-, M_{m|G^-})\leq O(m^{n-2})$ for dimension reasons, it is enough to show that $h^1(X, M_m)\leq O(m^{n-2})$. To this end, consider the Leray spectral sequence 
$$H^p(Y, R^{1-p}\vphi_{\ast}\OO_X(M_m))\Rightarrow  H^1(X, \OO_X(M_m)).$$ 
The terms $H^1(Y, \vphi_{*}\OO_X(M_m))= H^1(Y, \OO_{Y}(m))$ vanish for $m\gg0$ by Serre vanishing, so we need to prove 
\begin{equation}\label{equ:final}
h^0(Y, R^1\vphi_{\ast}\OO_{X}(M_m))\leq O(m^{n-2}).
\end{equation} 
Let $U\subseteq Y$ be the maximal open subset over which $\vphi$ is an isomorphism. By \cite[III.11.2]{Har77}, for each $m$ the sheaf $R^1\vphi_*\OO_{X}(M_m)$ is supported on the set $Y\setminus U$ of dimension at most $n-2$, hence $\chi(Y,R^1\vphi_{\ast}\OO_{X}(M_m))\leq O(m^{n-2})$. But by Serre vanishing again, all the higher cohomology groups of $R^1\vphi_{\ast}\OO_{X}(M_m)$ vanish for $m\gg0$, and this implies \eqref{equ:final}.
\end{proof}
	
The following is the main result of this section -- the geography of ample models.

\begin{thm}\label{thm:decomposition}
Let $X$ be a projective $\Q$-factorial variety, and let $\mathcal C\subseteq\Div_\R(X)$ be a rational polyhedral cone such that the ring $\mathfrak R=R(X,\mcal C)$ is finitely ge\-ne\-ra\-ted. Assume that $\Supp\mathfrak R$ contains a big divisor. Then there is a finite decomposition
\[\Supp\mathfrak R= \coprod \mcal{A}_i\]
into cones such that the following holds:
\begin{enumerate}
\item each $\overline{\mathcal A_i}$ is a rational polyhedral cone,
\item for each $i$, there exists a normal projective variety $X_i$ and a rational map $\vphi_i\colon X\dashto X_i$ such that $\varphi_i$ is the ample model for every $D\in\mathcal A_i$,
\item if $\mathcal A_j\subseteq\overline{\mathcal A_i}$, then there is a morphism $\varphi_{ij}\colon X_i\lto X_j$ such that the diagram 
\[
\xymatrix{ 
X \ar@{-->}[rr]^{\varphi_i} \ar@{-->}[dr]_{\varphi_j} & \quad & X_i
  \ar[dl]^{\varphi_{ij}}\\
\quad & X_j & \quad
}\]
commutes.
\item if $\mathcal A_i$ contains a big divisor, then $\varphi_i$ is a semiample model for every $D\in\overline{\mathcal{A}_i}$.
\end{enumerate}
\end{thm}

\begin{proof}
Let $\Supp\mathfrak R= \bigcup \mcal C_i$ be a finite rational polyhedral decomposition as in Theorem \ref{thm:ELMNP}, and let $\mcal{A}_i$ be the relative interior of $\mcal{C}_i$ for each $i$. We show that this is the required decomposition. 

Let $f\colon \widetilde{X} \lto X$ be a resolution and let $d$ be a positive integer as in Theorem \ref{thm:ELMNP}. For each $i$, fix $D_i\in \mcal{A}_i\cap \Div(X)$, and denote $M_i= \Mob f^*(dD_i)$ and $F_i=\Fix |f^*(dD_i)|$. Then $M_i$ is basepoint free, and let $\psi_i\colon \widetilde X \longrightarrow X_i$ be the semiample fibration associated to $M_i$. Let $\vphi_i\colon X \dashto X_i$ be the induced map.
\[
\xymatrix{ 
\widetilde X \ar[dr]^{\psi_i} \ar[d]_f & \\
X \ar@{-->}[r]^{\varphi_i} & X_i
}\]

\begin{cla}
Assume that $\mcal{A}_j \subseteq \overline{\mcal{A}_i}$, and let $C\subseteq \widetilde{X}$ be a curve such that $M_i\cdot C=0$. Then $M_j\cdot C=0$. In other words, all curves contracted by $\psi_i$ are contracted by $\psi_j$.
\end{cla}
Indeed, since $\mcal{A}_i$ is relatively open, there exist a divisor $D^\circ \in \mcal{A}_i\cap \Div(X)$ and positive integers $k_i,k_j,k^\circ$ such that $k_iD_i= k^\circ D^\circ+k_jD_j$. By the definition of $f$ and $d$, the divisor $M^\circ= \Mob f^*(dD^\circ)$ is basepoint free, and we have $k_iM_i= k^\circ M^\circ+k_jM_j$. In particular, if $C\subset \widetilde X$ is a curve such that $M_i\cdot C=0$, then $M^\circ \cdot C=M_j\cdot C=0$, which shows the claim. 

The claim immediately implies $\varphi_j= \vphi_{ij} \circ \vphi_{i}$ for some morphism $\vphi_{ij}\colon X_i \longrightarrow X_j$, which shows (3). In particular, when $i=j$ and since the divisors $D_i$ are arbitrary, this shows that the definition of $\vphi_i$ is independent of the choice of $D_i$ up to isomorphism.

Finally, we prove (2) and (4). For any $j$, pick an index $i$ such that $\mathcal A_j\subseteq\overline{\mcal{A}_i}$ and $\mcal{A}_i$ contains a big divisor, and let $E$ be the sum of all $f$-exceptional prime divisors. Since $\Mob(f^*(dD_i)+E)=M_i$ and $\Fix|f^*(dD_i)+E|=F_i+E$, the divisors $F_i$ and $E$ are $\psi_i$-exceptional by Lemma \ref{l:reid}, and in particular, $\vphi_i$ is a contraction.

Let $D$ be any divisor in $\mathcal A_j$; without loss of generality, we may assume that $D=D_j$. Since all functions $o_\Gamma$ are linear on $\overline{\mathcal A_i}$, we have $\Supp F_j\subseteq \Supp F_i$, hence $F_j$ is $\psi_i$-exceptional by the argument above. As $M_j=\psi_j^*\OO_{X_j}(1)$, by (3) we have 
\[
f^*(dD_j)=\psi_i^*(\vphi_{ij}^*\OO_{X_j}(1))+ F_j,
\]
and the divisor $(\vphi_i)_*(dD_j)=(\psi_i)_* M_j= \vphi_{ij}^*\OO_{X_{j}}(1)$ is basepoint free. We conclude that $\vphi_i$ is a semiample model for $D_j$, and $\varphi_j$ is the ample model for $D_j$.  
\end{proof}

An immediate corollary is the following result from \cite{HK00}; we prove the converse statement in the next section.  

\begin{cor}\label{lem:mds}
Let $X$ be a $\Q$-factorial projective variety. If $X$ is a Mori Dream Space, then its Cox ring is finitely generated.
\end{cor} 
\begin{proof}
We first show that the divisorial ring $R(X, \overline{\Mov} (X))$ is finitely generated. Indeed, with notation from Definition \ref{dfn:25}, we have that $\overline{\Mov}(X)=\bigcup\mathcal C_j$, where $\mathcal C_j=f_j^* \Nef(X_j)$, and hence it is enough to show that each ring $R(X,\mathcal C_j)\simeq R(X_j, \Nef(X_j))$ is finitely generated. But this is clear because each $\Nef(X_j)$ is spanned by finitely many semiample divisors. 

Let $\mathcal F_i$ be all the faces of all $\mcal C_j$ with the property that $\mcal F_i\subseteq\partial\overline{\Mov}(X)$ and $\mcal F_i\cap\B(X)\neq\emptyset$. Let $\vphi_i\colon X\dashto X_i$ be the ample models associated to interiors of $\mcal F_i$, cf.\ Theorem \ref{thm:decomposition}, and let $E_{ik}$ be the exceptional divisors of $\vphi_i$. Denote $\mcal D_i=\mcal F_i+\sum_k\R_+ E_{ik}$, and note that each $\mcal D_i$ is a rational polyhedral cone. 

We claim that $\overline{\Eff}(X)=\overline{\Mov}(X)\cup\bigcup_i\mcal D_i$. To see this, let $D\in\B(X)\backslash \overline{\Mov}(X)$ be a $\Q$-divisor. Then $P_\sigma(D)$ is a big $\Q$-divisor which belongs to $\partial\overline{\Mov}(X)$ by Lemma \ref{lem:projectionMov}, and hence the ring $R(X,D)$ is finitely generated by the above. There is a face $\mcal F_{i_0}$ which contains $P_\sigma(D)$ in its relative interior, and $\vphi_{i_0}$ is the ample model of $P_\sigma(D)$ by Theorem \ref{thm:decomposition}. The divisor $N_\sigma(D)$ is contracted by $\vphi_{i_0}$ by Lemma \ref{l:reid}, and thus $D\in\mcal D_{i_0}$. Therefore, we have $\B(X)\subseteq\overline{\Mov}(X)\cup\bigcup_i\mcal D_i$, and by taking closures we obtain $\overline{\Eff}(X)\subseteq\overline{\Mov}(X)\cup\bigcup_i\mcal D_i$. The converse inclusion is obvious.

In particular, the cone $\overline{\Eff}(X)$ is rational polyhedral, and $R(X,\overline{\Eff}(X))$ is a Cox ring of $X$. Fix an index $i$ and pick generators $G_1,\dots,G_p$ of $\mcal D_i$. It is enough to show that the ring $R(X;G_1,\dots,G_p)$ is finitely generated. The map $\vphi_i$ is a semiample model for each $G_\ell$ by Theorem \ref{thm:decomposition}(4), and thus $G_\ell=\vphi_i^*M_\ell+F_\ell$, where $M_\ell$ is a semiample $\Q$-divisor on $X_i$ and $F_\ell$ is $\vphi_i$-exceptional. But then $R(X;G_1,\dots,G_p)\simeq R(X_i;M_1,\dots,M_p)$, and the finite generation follows.
\end{proof}

The next theorem shows that in the classical setting of adjoint divisors, some of the ample models $X_i$ from Theorem \ref{thm:decomposition} are $\Q$-factorial. This is a known consequence of the classical Minimal Model Program \cite[Theorem 3.3]{HM09}, however here we obtain the result directly.

\begin{thm}\label{lem:adjoint}
Let $X$ be a projective $\Q$-factorial variety, and let $\Delta_1,\dots,\Delta_r$ be big $\Q$-divisors such that all pairs $(X,\Delta_i)$ are klt. Let
\[
\mathfrak R=R(X; K_X+\Delta_1, \dots, K_X+\Delta_r),
\]
and note that $\mathfrak R$ is finitely generated by Theorem \ref{thmA}. Assume that $\Supp\mathfrak R$ contains a big divisor. Then there exist a finite decomposition $\Supp\mathfrak R= \coprod \mcal{A}_i$ and maps $\vphi_i\colon X\dashto X_i$ as in Theorem \ref{thm:decomposition}, such that:
\begin{enumerate}
\item[(i)] if $\vphi_i$ is birational, then $X_i$ has rational singularities,
\item[(ii)] if the numerical classes of the elements of $\overline{\mcal{A}_i}$ span $N^1(X)_\R$, then $X_i$ is $\Q$-factorial.
\end{enumerate} 
\end{thm}
\begin{proof}
We assume the notation from the proof of Theorem \ref{thm:decomposition}. For (i), pick a big $\Q$-divisor $\Delta$ such that $(X,\Delta)$ is klt and $K_X+\Delta\in \mcal{A}_i$. Then $(X_i,(\vphi_i)_*\Delta)$ is also klt because $\vphi_i$ is $(K+\Delta)$-nonpositive, hence $X_i$ has rational singularities.

We now show (ii). Let $B$ be a Weil divisor on $X_i$, and let $\widetilde B$ be its proper transform on $\widetilde{X}$. As $\widetilde{X}$ is smooth, $\widetilde B$ is $\Q$-Cartier. Let $E_1,\dots,E_k$ be all the $f$-exceptional prime divisors on $\widetilde X$. Since $f$ is a resolution, we have
\begin{equation}\label{eq1}
\textstyle N^1(\widetilde X)_\R= f^*N^1(X)_\R \oplus \bigoplus_{j=1}^k \R[E_j].
\end{equation}
Let $B_1,\dots,B_r$ be integral divisors in $\mathcal A_i$ whose numerical classes generate $N^1(X)_\R$.  Then, by \eqref{eq1} there are rational numbers $p_j,r_j$ such that 
$$\widetilde B\equiv\sum p_jf^*(dB_j)+\sum r_jE_j.$$
Denote $M=\sum p_j\Mob f^*(dB_j)$ and $F=\sum p_j\Fix|f^*(dB_j)|+\sum r_jE_j$. By Theorem \ref{thm:decomposition}(4), there exist ample $\Q$-divisors $A_j$ on $X_i$ such that $\Mob f^*(dB_j)=\psi_i^*A_j$, hence $M\equiv_{X_i}0$. Therefore
$$\widetilde B-F\equiv_{X_i}0.$$
Observe that $\Supp F\subseteq\Supp (F_i+ \sum E_j)$, and that the divisor $F_i+ \sum E_j$ is $\psi_i$-exceptional by Lemma \ref{l:reid}. By (i) and by \cite[Proposition 12.1.4]{KM92}, there is a divisor $T\in\Div_\Q(X_i)$ such that $\widetilde B-F\sim_\Q \psi_i^*T$, and thus the divisor $B=(\psi_i)_*\widetilde B\sim_{\Q}T$ is $\Q$-Cartier.
\end{proof}

It is natural to ask whether the conclusion on $\Q$-factoriality from Theorem \ref{lem:adjoint} can be extended to the general situation of Theorem \ref{thm:decomposition}. We argue below that such a statement is, in general, not true, and we pin down precisely the obstacle to $\Q$-factoriality. The astonishing conclusion is that, in some sense, $\Q$-factoriality of ample models is essentially a condition on the numerical equivalence classes of the divisors in $\Supp\mathfrak R$.

With the notation from Theorem \ref{thm:decomposition}, what we are aiming for is the following statement. We would like to have a (possibly finer) decomposition $\Supp\mathfrak R= \coprod \mcal{N}_i$ together with birational maps $\vphi_i\colon X\dashto X_i$ such that $\vphi_i$ is an optimal model for every $D\in\mcal N_i$, and in particular, every $X_i$ is $\Q$-factorial. It is immediate that, if the numerical classes of the elements of $\mcal N_i$ span $N^1(X)_\R$, then $\vphi_i$ is also the ample model for every $D\in\mcal N_i$.

The following easy result gives us a necessary condition for the ample model of a big divisor to be $\Q$-factorial.

\begin{lem}\label{lem:gen divisors}
Let $X$ be a $\Q$-factorial projective variety, and let $D$ be a big $\Q$-divisor such that the ring $R(X, D)$ is finitely generated, and the map $\vphi \colon X \dashto \Proj R(X,D)$ is $D$-nonpositive. Let $D'$ be a $\Q$-divisor such that $D\equiv D'$.

Then the ring $R(X, D')$ is finitely generated if and only if the $\Q$-divisor $\vphi_*D'$ is $\Q$-Cartier.  
\end{lem}
\begin{proof}
If $R(X,D')$ is finitely generated, then by Lemma \ref{lem:equal proj}, $\vphi$ is equal to the map $X\dashto\Proj R(X,D')$ up to isomorphism. Therefore $\vphi_*D'$ is ample, and in particular $\Q$-Cartier. 

For the converse implication, denote $Y=\Proj R(X,D)$ and let $(p,q)\colon W\lto X\times Y$ be a resolution of $\vphi$. By Lemma \ref{cor:negative}, we have
\[
 p^*(D-D') = q^*\vphi_*(D-D'),
\]
hence $\vphi_*D\equiv\vphi_*D'$. Since $\vphi_*D$ is ample, so is $\vphi_*D'$, hence the ring $R(Y,\varphi_*D')$ is finitely generated. By Lemma \ref{l:reid}, the divisor $E=p^*D-q^*\varphi_*D$ is effective and $q$-exceptional, and since $E=p^*D'-q^*\varphi_*D'$, we have $R(X,D')\simeq R(Y,\varphi_*D')$. 
\end{proof}

Therefore, in the notation of Lemma \ref{lem:gen divisors}, if the ample model of $D$ is $\Q$-factorial, then the ring $R(X,D')$ is finitely generated for every $\Q$-divisor $D'$ in the numerical class of $D$. This motivates the following key definition.

\begin{dfn}\label{dfn:gen}
Let $X$ be a $\Q$-factorial projective variety. We say that a divisor $D\in\Div_\Q(X)$ is \emph{gen} if for all $\Q$-Cartier $\Q$-divisors $D'\equiv D$, the section ring $R(X,D')$ is finitely generated. 
\end{dfn}

There are three main examples of gen divisors of interest to us:
\begin{enumerate}
\item[(i)] ample $\Q$-divisors are gen,
\item[(ii)] every adjoint divisor $K_X+\Delta+A$ is gen, where $A$ is an ample $\Q$-divisor on $X$, and the pair $(X,\Delta)$ is klt; indeed, this follows from Theorem \ref{thmA},
\item[(iii)] if $\Pic(X)_\Q=N^1(X)_\Q$, then every divisor with a finitely generated section ring is gen.
\end{enumerate}
As we show in Section \ref{sec:DMMP}, having lots of gen divisors is essentially equivalent to being able to run a Minimal Model Program. We have seen above that this is a necessary condition for the models to be optimal, and in particular $\Q$-factorial. We show in Theorem \ref{thm:scalingbig} that, remarkably, this is also a sufficient condition. This, together with (ii) and (iii), explains \emph{precisely} why we are able to run the MMP for adjoint divisors and on Mori Dream Spaces, and the details are worked out in Corollaries \ref{cor:lt models} and \ref{cor:MDS}.

We conclude this section with an example where all the conditions of Theorem \ref{thm:decomposition} are satisfied, but the absence of gen divisors implies that there is no decomposition of $\Supp \mathfrak R$ into regions of divisors that share an optimal model. In particular, we cannot run the MMP as explained in Section \ref{sec:DMMP}, and therefore the conditions from Theorems \ref{lem:nullflip} are \ref{thm:scalingbig} are not only sufficient, but they are optimal. The example shows that the finite generation of a divisorial ring in itself is not sufficient to perform the Minimal Model Program.

\begin{exa}\label{exa:notgen}
Let $X$, $L_1$ and $L_2$ be as in Example \ref{exa1}, and note that $X$ is a smooth surface with $\dim N^1(X)_\R=2$. We show that there exist a big divisor $D$ and an ample divisor $A$ on $X$ such that the ring $\mathfrak R=R(X;D,A)$ is finitely generated, the divisor $L_1$ belongs to the interior of the cone $\Supp\mathfrak R=\R_+D+\R_+A$, and \emph{none} of the divisors in the cone $\R_+D+\R_+L_1\subseteq\Supp\mathfrak R$ is gen. In particular, we cannot perform the MMP for $D$.

We first claim that there exists an irreducible curve $C$ on $X$ such that
\begin{equation}\label{equ:curve}
L_1\cdot C=0\quad\text{and}\quad C^2<0.
\end{equation}
Indeed, since $L_1$ is semiample but not ample, there exists an irreducible curve $C\subseteq X$ such that $L_1\cdot C=0$. Since $L_1$ is big and nef, we have $L_1^2>0$, so the Hodge index theorem then implies $C^2<0$. 

Now, set $D = L_1+C$. Since $\dim N^1(X)_\R=2$ and $D$ is not nef, it is immediate that there exists an ample divisor $A$ on $X$ such that $L_1\in\R_+D+\R_+A$. In order to show that $\mathfrak R$ is finitely generated, it is enough to show that the rings $R(X;D,L_1)$ and $R(X;L_1,A)$ are finitely generated, and this latter ring is finitely generated since both $L_1$ and $A$ are semiample. 

For $k_1,k_2\in\N$, consider the divisor $D_{k_1,k_2}=k_1D+k_2L_1=(k_1+k_2)L_1+k_1C$. Then \eqref{equ:curve} implies that $P_\sigma(D_{k_1,k_2})=(k_1+k_2)L_1$, and therefore $H^0(X,D_{k_1,k_2})\simeq H^0(X,(k_1+k_2)L_1)$. Hence the ring
$$ R(X;D,L_1)\simeq R(X;L_1,L_1)$$
is finitely generated. 

Finally, note that $D_{k_1,k_2}\equiv (k_1+k_2)L_2 + k_1C$, and that $P_\sigma((k_1+k_2)L_2 + k_1C)=(k_1+k_2)L_2$. Therefore the ring
\[
 R(X,(k_1+k_2)L_2 + k_1C) \simeq R(X,(k_1+k_2)L_2)
\]
is not finitely generated, thus the divisor $D_{k_1,k_2}$ is not gen.
\end{exa}

\begin{rem}
The notion of genness is a very subtle one. For instance, every $\Q$-divisor $D$ with $\kappa_\sigma(D)=0$ is gen (for the definition and properties of $\kappa_\sigma$ see \cite{Nak04}). Indeed, for every $\Q$-divisor $D'\equiv D$ we have $\kappa(D')\leq\kappa_\sigma(D')=0$, hence the ring $R(X,D)$ is isomorphic to either $\C$ or to the polynomial ring $\C[T]$.
%
\end{rem}
\section{Running the $D$-MMP}\label{sec:DMMP}

Let $X$ be a projective $\Q$-factorial variety, and let $\mathcal C\subseteq\Div_\R(X)$ be a rational polyhedral cone such that the divisorial ring $\mathfrak R=R(X,\mathcal C)$ is finitely generated. Then by Theorem \ref{thm:decomposition} we know that $\Supp\mathfrak R$ has a decomposition into finitely many rational polyhedral cones giving the geography of ample models associated to $\mathfrak R$.

In this section we explain how, when all divisors in the interior of $\Supp\mathfrak R$ are gen, the aforementioned decomposition can be refined to give a geography of optimal models. As indicated in the previous sections, the main technical obstacle is to prove $\Q$-factoriality of models, and this is the point where the gen condition on divisors plays a crucial role.

We assume that $\Supp\mathfrak R$ contains an ample divisor, and fix a divisor $D\in\Supp\mathfrak R$. Then we can run the Minimal Model Program for $D$ as follows.

We define a certain finite rational polyhedral decomposition $\mathcal C=\bigcup\mathcal N_i$ in Theorem \ref{thm:scalingbig}. If $D$ is not nef, we show in Theorem \ref{lem:nullflip} that there is a $D$-negative birational map $\varphi\colon X\dashto X^+$ such that $X^+$ is $\Q$-factorial, and $\vphi$ is elementary -- this corresponds to contractions of extremal rays in the classical MMP. We also show that there is a rational polyhedral subcone $D\in\mathcal C'\subseteq\mathcal C$ which is a union of \emph{some, but not all} of the cones $\mathcal N_i$, such that $R(X,\mathcal C')\simeq R(X^+,\vphi_*\mathcal C')$ and the cone $\varphi_*\mathcal C'\subseteq\Div_\R(X^+)$ contains an ample divisor. Now we replace $X$ by $X^+$, $D$ by $\varphi_*D$, and $\mathcal C$ by $\vphi_*\mathcal C'$, and we repeat the procedure. Since there are only finitely many cones $\mathcal N_i$, this process must terminate with a variety $X_D$ on which the proper transform of $D$ is nef, and this is the optimal model for $D$. It is then automatic that $X_D$ is also an optimal model for all divisors in the cone $\mathcal N_{i_0}\ni D$. The details are given in Theorem \ref{thm:scalingbig}.

In the context of adjoint divisors and the classical MMP, we can additionally \emph{direct} the MMP by an ample $\Q$-divisor $A$ on $X$, as in \cite{CoL10}. The proofs of Theorems \ref{lem:nullflip} and \ref{thm:scalingbig} can be easily modified to obtain the {\em $D$-MMP with scaling of $A$}, however we do not pursue this here.

First we define elementary contractions.

\begin{dfn}
A birational contraction $\varphi\colon X\dashto Y$ between normal projective varieties is \emph{elementary} if it not an isomorphism, and it is either an isomorphism in codimension $1$, or a morphism whose exceptional locus is a prime divisor on $X$.
\end{dfn}

The following theorem is the key result: it shows that in our situation elementary contractions exist.

\begin{thm}\label{lem:nullflip}
Let $X$ be a projective $\Q$-factorial variety and let $\mathcal C\subseteq\Div_\R(X)$ be a rational polyhedral cone. Denote by $\pi\colon\Div_\R(X)\lto N^1(X)_\R$ the natural projection. Assume that the ring $\mathfrak R=R(X,\mcal C)$ is finitely ge\-ne\-ra\-ted, that $\Supp\mathfrak R$ contains an ample divisor, that $\pi(\Supp\mathfrak R)$ spans $N^1(X)_\R$, and that every divisor in the interior of $\Supp\mathfrak R$ is gen. Let $\Supp\mathfrak R=\bigcup\mcal C_i$ be a decomposition as in Theorem \ref{thm:ELMNP}. Let $D\in\Supp\mathfrak R$ be a $\Q$-divisor which is not nef. 
Then:
\begin{enumerate}
\item the cone $\Supp\mathfrak R\cap\pi^{-1}\big(\Nef(X)\big)$ is rational polyhedral, and every element of this cone is semiample,
\item there exists a rational hyperplane $\mathcal H\subseteq N^1(X)_\R$ which intersects the interior of $\pi(\Supp\mathfrak R)$ and contains a codimension $1$ face of $\pi(\Supp\mathfrak R)\cap\Nef(X)$, such that $\pi(D)$ and $\Nef(X)$ are on the opposite sides of $\mathcal H$,
\item let $\mathcal W\subseteq N^1(X)_\R$ be the half-space bounded by $\mathcal H$ which does not contain $\Nef(X)$, and let $\mathcal C'=\Supp\mathfrak R\cap\pi^{-1}(\mathcal W)$. Then there exists a $\Q$-factorial projective variety $X^+$ together with an elementary contraction $\vphi\colon X\dashto X^+$, such that $\varphi$ is $W$-nonpositive for every $W\in\mathcal C'$, and it is $W$-negative for every $W\in\mathcal C'\backslash\pi^{-1}(\mathcal H)$, 
\item we have $R(X,\mathcal C')\simeq R(X^+,\mathcal C^+)$, where $\mathcal C^+=\varphi_*\mathcal C'\subseteq\Div_\R(X^+)$, and $\mcal C^+$ contains an ample divisor,
\item for every cone $\mcal C_i$ and for every geometric valuation $\Gamma$ over $X$, the function $o_\Gamma$ is linear on the cone $\vphi_*(\mcal C'\cap\mcal C_i)\subseteq \mcal C^+$.
\end{enumerate}
\end{thm}
\begin{figure}[htb]
\begin{center}
\includegraphics[width=0.4\textwidth]{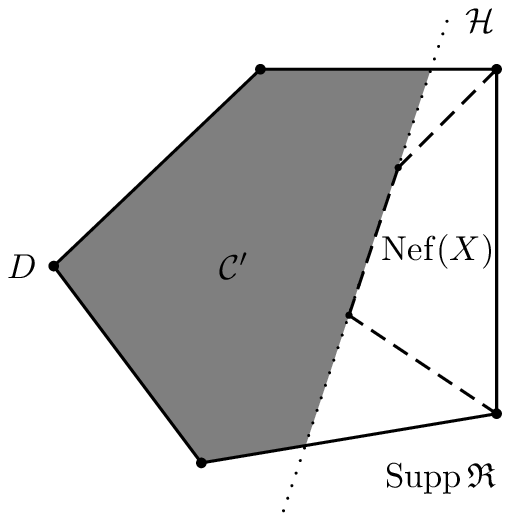}
\end{center}
\end{figure}

\begin{proof}
{\em Step 1.}
The statement (1) follows immediately from Corollary \ref{cor:5}, statement (4) follows from (3) by Remark \ref{rem:1}(i) and from the construction below, while (5) follows from (3) by Lemma \ref{lem:2}. To show (2), let $\alpha$ be any ample class in the interior of $\pi(\Supp\mathfrak R)\subseteq N^1(X)$, and let $\beta$ be the intersection of the segment $[\pi(D),\alpha]$ with $\partial\Nef(X)$. Then $\beta$ lies in the interior of $\pi(\Supp\mathfrak R)$, and by (1) there is a rational codimension $1$ face of $\pi(\Supp\mathfrak R)\cap\Nef(X)$ containing $\beta$. We define $\mathcal H$ to be the rational hyperplane containing that face.\\[2mm]
{\em Step 2.}
It remains to show (3). By Corollary \ref{cor:5}, there are cones $\mathcal C_j\nsubseteq\pi^{-1}\big(\Nef(X)\big)$ and $\mathcal C_k\subseteq\pi^{-1}\big(\Nef(X)\big)$ such that $\dim\pi(\mathcal C_j)=\dim\pi(\mathcal C_k)=\dim N^1(X)_\R$ and $\pi(\mathcal C_j)\cap\pi(\mathcal C_k)\subseteq\mathcal H$; denote $\mathcal C_{jk}=\mathcal C_j\cap\mathcal C_k$. Let $\varphi\colon X\dashrightarrow X^+$ and $\theta\colon X\dashrightarrow Y$ be the ample models associated to relative interiors of $\mathcal C_j$ and $\mathcal C_{jk}$ as in the proof of Theorem \ref{thm:decomposition}, and note that $\theta$ is a morphism by (1) since $\mathcal C_{jk}\subseteq\Supp\mathfrak R\cap\pi^{-1}\big(\Nef(X)\big)$. Then, by Theorem \ref{thm:decomposition}(3), there is a morphism $\theta^+\colon X^+\longrightarrow Y$ such that the diagram 
\[
\xymatrix{ 
X \ar@{-->}[rr]^{\varphi} \ar[dr]_\theta & \quad & X^+ \ar[dl]^{\theta^+}\\
\quad & Y & \quad
}\]
is commutative. The following is the key claim:
\begin{cla}\label{claim:pullback}
Let $F$ be an $\R$-divisor on $X$ such that $\pi(F)\in \mcal{H}$. Then $F\sim_{\R} \theta^*F_Y$ for some $F_Y\in \Div_{\R}(Y)$. If additionally  $\pi(F)\in\pi(\mathcal C_{jk})$, then $F_Y$ is ample. In particular, a curve $C$ is contracted by $\theta$ if and only if $C\cdot\delta=0$ for every $\delta\in\mathcal H$.
\end{cla}
Pick $\Q$-divisors $B_1,\dots,B_r$ in $\mcal{C}_{jk}$ and nonzero real numbers $\lambda_i$ such that $\pi(B_i)$ span $\mcal{H}$ and $\pi(F)= \sum \lambda_i \pi(B_i)$. We may assume that $\lambda_i\geq0$ for all $i$ when $\pi(F)\in\pi(\mcal C_{jk})$.  Hence, there is a $\Q$-divisor $B'_1\equiv B_1$ such that $F= \lambda_1 B'_1 +\sum_{i\geq 2} \lambda_i B_i$. Note that, by the definition of $\theta$, there are ample divisors $A_i$ on $Y$ such that $B_i\sim_\Q \theta^*A_i$ for all $i\geq 2$.

Since $B_1$ is gen, the ring $R(X, B'_1)$ is finitely generated, and therefore $B_1'$ is semiample by Lemma \ref{lem:ords}(2) as it is nef and big. Denote by $\theta'\colon X \lto Y'$ the semiample fibration associated to $B_1'$. By Lemma \ref{lem:equal proj}, there is an isomorphism $\eta \colon Y\lto Y'$ such that $\theta'= \eta\circ \theta$. Since $B_1'\sim_\Q(\theta')^*A_1'$ for an ample divisor $A_1'$ on $Y'$, we have $B_1'\sim_\Q \theta^*A_1$, where $A_1=\eta^*A_1'$. Therefore $F\sim_\R \theta^*(\sum \lambda_i A_i)$, which proves the claim.
\\[2mm]
{\em Step 3.}
We next show that $X^+$ is $\Q$-factorial.

Consider a Weil divisor $P^+$ on $X^+$, and let $P$ be its proper transform on $X$. Since $X$ is $\Q$-factorial, the divisor $P$ is $\Q$-Cartier. Since $\dim \pi (\mcal{C}_j)= \dim N^1(X)_\R$, there exist a $\Q$-divisor $G \in \mcal{C}_j$ and  $\alpha\in \Q$ such that $\pi(P+\alpha G)\in \mcal{H}$. By Claim \ref{claim:pullback}, there exists $M\in \Div_{\Q}(Y)$ such that $P+\alpha G\sim_{\Q}\theta^* M$. Let $(p,q)\colon \widetilde{X}\longrightarrow X \times X^+$ be a resolution of $\vphi$. By the definition of $\vphi$ and by Theorem \ref{thm:decomposition}, there is an ample $\Q$-divisor $A$ on $X^+$ and a $q$-exceptional $\Q$-divisor $E$ on $\widetilde X$ such that $p^*G= q^*A+E$. It follows that $p^*P\sim_\Q (\theta\circ p)^*M-\alpha(q^*A+E)= (\theta^+\circ q)^*M-\alpha q^*A-\alpha E$. Since $\vphi$ is a contraction, we have $P^+=q_*p^*P$, and therefore the divisor
$$P^+\sim_\Q(\theta^+)^*M-\alpha A$$
is $\Q$-Cartier.\\[2mm]
{\em Step 4.}
In this step we show that $\vphi$ is an elementary map. 

If $\theta$ is an isomorphism in codimension $1$, then so are $\vphi$ and $\theta^+$ as $\vphi$ is a contraction. 

Hence, we may assume that there exists a $\theta$-exceptional prime divisor $E$. Let $C$ be a curve contracted by $\theta$, and let $R$ be a ray in $N_1(X)_\R$ orthogonal to the hyperplane $\mathcal H$. Then the class of $C$ belongs to $R$ by Claim \ref{claim:pullback}, and so $E\cdot R<0$ by Lemma \ref{lem:negativity}. In particular, we have $E\cdot C<0$, thus $C\subseteq E$, and the exceptional locus of $\theta$ equals $E$. Therefore, $\theta$ is an elementary contraction. 

Observe that $\pi(E)$ and $\Nef(X)$ lie on opposite sides of $\mathcal H$. This implies that there is a $\Q$-divisor $G_E$ in the relative interior of $\mathcal C_j$ such that $\pi(G_E-E)$ belongs to the relative interior of $\mathcal C_{jk}$. Then, by Claim \ref{claim:pullback}, there exists an ample divisor $M_E\in\Div_\Q(Y)$ such that $G_E-E\sim_\Q\theta^*M_E$, and thus
\begin{equation}\label{equ:isomorphism}
H^0(X,mG_E)\simeq H^0(X,m \theta^*M_E)
\end{equation}
for every positive integer $m$. Since $\vphi$ is the map $X\dashto\Proj R(X,G_E)$ by definition, we may assume that $X^+=Y$ and $\vphi=\theta$ by \eqref{equ:isomorphism}, which shows that $\vphi$ is an elementary contraction.\\[2mm]
{\em Step 5.}
The only thing left to prove is the last statement in (3). For $W\in\mcal C'$, there exists an $\R$-divisor $G_W\in\mcal C_j$ such that $\pi(W-G_W)\in \mcal{H}$. Thus $W\equiv_Y G_W$ by Claim \ref{claim:pullback}. Since $\varphi$ is $G_W$-nonpositive by Theorem \ref{thm:decomposition}(4), this implies that $\varphi$ is $W$-nonpositive by Corollary \ref{cor:negative}. If $\varphi$ is an isomorphism in codimension $1$, it is automatic that it is then also $W$-negative.

If $W\in\mathcal C'\backslash\pi^{-1}(\mathcal H)$ and $\vphi$ contracts a divisor $E$, there exists a positive rational number $\lambda$ such that $\pi(W-\lambda E)\in\mathcal H$. Again by Claim \ref{claim:pullback}, and since $X^+=Y$ and $\varphi=\theta$, there is a divisor $M_W\in\Div_\R(X^+)$ such that $W-\lambda E\sim_\R\vphi^*M_W$. But then it is clear that $\vphi$ is $W$-negative.
\end{proof}


The following is the main result of this paper -- the geography of optimal models.

\begin{thm}\label{thm:scalingbig}
Let $X$ be a projective $\Q$-factorial variety, and let $\mathcal C\subseteq\Div_\R(X)$ be a rational polyhedral cone. Denote by $\pi\colon\Div_\R(X)\lto N^1(X)_\R$ the natural projection. Assume that the ring $\mathfrak R=R(X,\mcal C)$ is finitely ge\-ne\-ra\-ted, that $\Supp\mathfrak R$ contains an ample divisor, that $\pi(\Supp\mathfrak R)$ spans $N^1(X)_\R$, and that every divisor in the interior of $\Supp\mathfrak R$ is gen. 

Then for any $\Q$-divisor $D\in\mathcal C$, we can run a $D$-MMP which terminates.

Furthermore, there is a finite decomposition
$$\Supp\mathfrak R= \coprod \mcal{N}_i$$
into cones having the following properties:
\begin{enumerate}
\item each $\overline{\mathcal N_i}$ is a rational polyhedral cone,
\item for each $i$, there exists a $\Q$-factorial projective variety $X_i$ and a birational contraction $\vphi_i\colon X\dashto X_i$ such that $\varphi_i$ is an optimal  model for every divisor in $\mathcal N_i$,
\item every element of the cone $(\vphi_i)_*\mcal N_i$ is semiample.
\end{enumerate}
\end{thm}

\begin{proof}
Denote by $V\subseteq\Div_\R(X)$ the minimal vector space containing $\mcal C$, and define $\mcal C^1=\Supp\mathfrak R$. Let $\mcal C^1= \bigcup_{i\in I_1}\mcal C_i^1$ be the rational polyhedral decomposition as in Theorem \ref{thm:ELMNP}. By subdividing $\mcal C^1$ further, we may assume that the following property is satisfied:
\begin{itemize}
\item[($\natural$)] let $\mathcal G\subseteq V$ be any hyperplane which contains a codimension $1$ face of some $\mathcal C_{i_0}^1$. Then every $\mathcal C_i^1$ is contained in one of the two half-spaces of $V$ bounded by $\mathcal G$.
\end{itemize}
For each $i\in I_1$, let $\mathcal N_i$ be the relative interior of $\mathcal C_i$. We claim that $\mcal C^1=\coprod_{i\in I_1} \mcal{N}_i$ is the desired decomposition.

Let $D$ be a point in some $\mathcal N_{i_0}$. If $D$ is nef, then every divisor in $\mathcal N_{i_0}$ is semiample by Corollary \ref{cor:5}, so the theorem follows.

Therefore, we may assume that $D$ is not nef. Denote $Y_1=X$ and $D_1=D$. We show that there exists a $D_1$-MMP which terminates.

By Theorem \ref{lem:nullflip}, the cone $\mcal C^1\cap \pi^{-1}\bigl(\Nef (Y_1)\bigr)$ is rational polyhedral. Let $\mathcal H\subseteq N^1(Y_1)_\R$ be a rational hyperplane as in Theorem \ref{lem:nullflip}, and let $\mathcal C_{\ell}^1$, for $\ell\in I_2\subsetneq I_1$, be those cones for which $\pi(\mathcal C_{\ell}^1)$ and $\pi(D)$ are on the same side of $\mathcal H$, cf.\ ($\natural$). Let $f_1\colon Y_1\dashrightarrow Y_2$ be an elementary map as in Theorem \ref{lem:nullflip}(3), and denote $D_2=(f_1)_*D_1$. Define rational polyhedral cones $\mathcal C_{\ell}^2=(f_1)_*\mathcal C^1_{\ell}\subseteq\Div_{\R}(Y_2)$, and set 
\begin{equation}\label{equ:decomposition}
\textstyle\mathcal C^2= \bigcup_{\ell\in I_2} \mathcal C^2_{\ell}.
\end{equation}
Then the ring $\mathfrak R^2=R(Y_2,\mathcal C^2)$ is finitely generated by Theorem \ref{lem:nullflip}(4). By Theorem \ref{lem:nullflip}(5), the relation \eqref{equ:decomposition} gives a decomposition of $\mcal C^2$ as in Theorem \ref{thm:ELMNP}. Also note that $(f_1)_*(\mcal N_{i_0})\subseteq\mcal C^2$. 

In this way we construct a sequence of divisors $D_p$ on $\Q$-factorial varieties $Y_p$. Since the size of the index sets $I_p$ drops with each step,  this process must terminate with a model $X_{p_0}$ on which the divisor $D_{p_0}$ is nef. Similarly as above, $X_{p_0}$ is an optimal model for all divisors in $\mathcal N_{i_0}$, and the proper transform on $Y_{p_0}$ of every element of $\mcal N_{i_0}$ is semiample.
\end{proof}

\begin{cor}\label{cor:stepsMMP}
Let $X$ be a projective $\Q$-factorial variety, let $S_1,\dots,S_p$ be distinct prime divisors on $X$, denote $V=\sum_{i=1}^p\mathbb R S_i\subseteq \Div_{\mathbb R}(X)$, and let $A$ be an ample $\mathbb Q$-divisor on $X$. Let $\mcal C\subseteq\mcal L(V)$ be a rational polytope such that for every $\Delta\in \mcal C$, the pair $(X,\Delta)$ is klt.

Then there exists a positive integer $M$ such that for every $\Delta\in\mcal C\cap\mathcal E_A(V)$, there is a $(K_X+\Delta)$-MMP consisting of at most $M$ steps.
\end{cor}
\begin{proof}
By enlarging $V$ and $\mcal C$, we may assume that the numerical classes of the elements of $\mcal C\cap\mathcal E_A(V)$ span $N^1(X)_\R$. The set $\mcal C\cap\mcal E_A(V)$ is a rational polytope by Corollary \ref{cor:effective}, and let $B_1,\dots,B_r$ be its vertices. Choose a positive integer $\lambda\gg0$ such that all $K_X+A+B_i+\lambda A$ are ample. Denote 
$$\mathcal D=\sum\R_+(K_X+A+B_i)+\sum\R_+(K_X+A+B_i+\lambda A).$$
Then the ring $\mathfrak R=R(X,\mathcal D)$ is finitely generated by Theorem \ref{thmA}, and we have $\R_+(K_X+A+\mcal C\cap\mathcal E_A(V))\subseteq\Supp\mathfrak R$. Let $\Supp\mathfrak R= \coprod_{i=1}^N \mcal{N}_i$ be the decomposition as in Theorem \ref{thm:scalingbig}. Then it is immediate from the proof of Theorem \ref{thm:scalingbig} that we can set $M=N$.
\end{proof}

The following corollary is finiteness of models, cf.\ \cite[Lemma 7.1]{BCHM}.

\begin{cor}\label{cor:lt models}
Let $X$ be a projective $\Q$-factorial variety, let $S_1,\dots,S_p$ be distinct prime divisors on $X$, denote $V=\sum_{i=1}^p\mathbb R S_i\subseteq \Div_{\mathbb R}(X)$, and let $A$ be an ample $\mathbb Q$-divisor on $X$. Let $\mcal C\subseteq\mcal L(V)$ be a rational polytope such that for every $\Delta\in \mcal C$, the pair $(X,\Delta)$ is klt.

Then there are finitely many rational maps $\varphi_i\colon X \dashrightarrow Y_i$, with the property that if $\Delta\in\mcal C\cap\mathcal E_A(V)$, then there is an index $i$ such that $\varphi_i$ is a log terminal model of $K_X +\Delta$.
\end{cor}
\begin{proof}
By enlarging $V$ and $\mcal C$, we may assume that the numerical classes of the elements of $\mcal C\cap\mathcal E_A(V)$ span $N^1(X)_\R$, and that there exists a divisor $B\in\mcal C\cap\mathcal E_A(V)$ such that $K_X+A+B$ is ample. The ring $R(X,\R_+(K_X+A+\mcal C\cap\mathcal E_A(V)))$ is finitely generated by Corollary \ref{cor:effective}, so the result follows immediately from Theorem \ref{thm:scalingbig}.
\end{proof}

Finally, we recover one of the main results of \cite{HK00}.

\begin{cor}\label{cor:MDS}
Let $X$ be a $\Q$-factorial projective variety such that $\Pic(X)_\Q = N^1(X)_\Q$. Then $X$ is a Mori Dream Space if and only if its Cox ring is finitely generated.

In particular, if $(X,\Delta)$ is a klt log Fano pair, then $X$ is a Mori Dream Space.
\end{cor}
\begin{proof}
Let $D_1, \dots , D_r$ be a basis of $\Pic(X)_\Q$ such that $\overline{\Eff}(X)\subseteq \sum \R_+D_i$. The associated divisorial ring $\mathfrak R=R(X; D_1, \dots , D_r)$ is a Cox ring of $X$. Corollary \ref{lem:mds} shows that if $X$ is a Mori Dream Space, then $\mathfrak R$ is finitely generated.  We now prove the converse statement.

Assume that $\mathfrak R$ is finitely generated, and let $\Supp \mathfrak R= \coprod_{i=1}^N \mcal{N}_i$ be the decomposition from Theorem \ref{thm:scalingbig}. Then $\Nef(X)$ is the span of finitely many semiample divisors by Corollary \ref{cor:5}, and by the definition of the sets $\mathcal N_i$ and by Corollary \ref{cor:5}, there is a set $I\subseteq\{1,\dots,N\}$ such that $\overline{\Mov}(X)=\bigcup_{i\in I}\overline{\mathcal N_i}$. By taking a smaller index set $I$, we may assume that $\dim \overline{\mcal{N}_i}= \dim N^1(X)_\R$ for all $i\in I$. For $i\in I$, let $\vphi_i\colon X\dashto X_i$ be the maps as in Theorem \ref{thm:scalingbig}. Then $\overline{\mathcal N_i}\subseteq\varphi_i^*\big(\Nef(X_i)\big)$, and hence $\overline{\Mov}(X)\subseteq\bigcup_{i\in I}\varphi_i^*\big(\Nef(X_i)\big)$. Each $\varphi_i$ is an optimal  model for every divisor in $\mcal N_i$, thus each $\vphi_i$ is an isomorphism in codimension $1$. Therefore, the ring $R(X_i;(\vphi_i)_*D_1,\dots,(\vphi_i)_*D_r)$ is a Cox ring of $X_i$, and it is finitely generated since it is isomorphic to $\mathfrak R$. In particular, every $\Nef(X_i)$ is spanned by finitely many semiample divisors by above, and hence $\overline{\Mov}(X)\supseteq\bigcup_{i\in I}\varphi_i^*\big(\Nef(X_i)\big)$. This shows that $X$ is a Mori Dream Space.

Now, if $(X, \Delta)$ is a klt log Fano pair, then $H^i(X, \OO_X)= 0$ for all $i>0$ by Kawamata-Viehweg vanishing. The long exact sequence in cohomology associated to the exponential sequence
\[
0 \longrightarrow \Z \longrightarrow \OO_X \longrightarrow \OO^*_X \longrightarrow 0
\]
shows that $\Pic(X)_\Q= N^1(X)_{\Q}$. Let $D_1,\dots,D_r$ be a basis of $\Pic(X)_\Q$ such that $\overline{\Eff}(X)\subseteq \sum \R_+D_i$, and pick a rational number $0<\varepsilon\ll1$ such that $A_i=\varepsilon D_i-(K_X+\Delta)$ is ample for every $i$. Then the ring $R(X; \varepsilon D_1, \dots,\varepsilon D_r)=R(X; K_X+\Delta+A_1, \dots, K_X+\Delta+A_r)$ is finitely generated by Theorem \ref{thmA}, and hence the Cox ring $R(X; D_1, \dots, D_r)$ is finitely generated by Lemma \ref{lem:3}.  
\end{proof}

\bibliographystyle{amsalpha}

\bibliography{biblio}

\newcommand{\etalchar}[1]{$^{#1}$}
\providecommand{\bysame}{\leavevmode\hbox to3em{\hrulefill}\thinspace}
\providecommand{\MR}{\relax\ifhmode\unskip\space\fi MR }
\providecommand{\MRhref}[2]{%
  \href{http://www.ams.org/mathscinet-getitem?mr=#1}{#2}
}
\providecommand{\href}[2]{#2}
\begin{thebibliography}{BCHM10}

\bibitem[BCHM10]{BCHM}
C.~Birkar, P.~Cascini, C.~D. Hacon, and J.~M\textsuperscript{c}Kernan,
  \emph{Existence of minimal models for varieties of log general type}, J.
  Amer. Math. Soc. \textbf{23} (2010), no.~2, 405--468.

\bibitem[BKS04]{BKS}
Th. Bauer, A.~K\"uronya, and T.~Szemberg, \emph{Zariski decompositions,
  volumes, and stable base loci}, J. f\"ur die reine und angew. Math.
  \textbf{576} (2004), 209--233.

\bibitem[CL12a]{CaL10}
P.~Cascini and V.~Lazi\'c, \emph{New outlook on the {M}inimal {M}odel
  {P}rogram, {I}}, Duke Math. J. \textbf{161} (2012), no.~12, 2415--2467.

\bibitem[CL12b]{CoL10}
A.~Corti and V.~Lazi\'c, \emph{New outlook on the {M}inimal {M}odel {P}rogram,
  {II}}, Math. Ann. (2012), DOI:10.1007/s00208-012-0858-1.

\bibitem[Deb01]{Deb01}
O.~Debarre, \emph{Higher-dimensional algebraic geometry}, Universitext,
  Springer-Verlag, New York, 2001.

\bibitem[ELM{\etalchar{+}}06]{ELMNP}
L.~Ein, R.~Lazarsfeld, M.~Musta{\c{t}}{\u{a}}, M.~Nakamaye, and M.~Popa,
  \emph{Asymptotic invariants of base loci}, Ann. Inst. Fourier (Grenoble)
  \textbf{56} (2006), no.~6, 1701--1734.

\bibitem[Har77]{Har77}
R.~Hartshorne, \emph{Algebraic geometry}, Graduate Texts in Mathematics,
  vol.~52, Springer-Verlag, New York, 1977.

\bibitem[HK00]{HK00}
Y.~Hu and S.~Keel, \emph{Mori dream spaces and {GIT}}, Michigan Math. J.
  \textbf{48} (2000), 331--348.

\bibitem[HM09]{HM09}
C.~D. Hacon and J.~M\textsuperscript{c}Kernan, \emph{The {S}arkisov program},
  arXiv:0905.0946\setbox0=\hbox{2009}.

\bibitem[K{\etalchar{+}}92]{Kol92}
J.~Koll{\'a}r et~al., \emph{Flips and abundance for algebraic threefolds},
  Ast\'erisque 211, Soc.\ Math.\ France, Paris, 1992.

\bibitem[Kaw88]{Kaw88}
Y.~Kawamata, \emph{Crepant blowing-up of {$3$}-dimensional canonical
  singularities and its application to degenerations of surfaces}, Ann. of
  Math. (2) \textbf{127} (1988), no.~1, 93--163.

\bibitem[Kaw97]{Kaw97}
\bysame, \emph{On the cone of divisors of {C}alabi-{Y}au fiber spaces},
  Internat.\ J.\ Math. \textbf{8} (1997), 665--687.

\bibitem[KM92]{KM92}
J.~Koll{\'a}r and S.~Mori, \emph{Classification of three-dimensional flips}, J.
  Amer. Math. Soc. \textbf{5} (1992), no.~3, 533--703.

\bibitem[Laz04]{Laz04}
R.~Lazarsfeld, \emph{Positivity in algebraic geometry. {I}, {II}}, Ergebnisse
  der Mathematik und ihrer Grenzgebiete, vol. 48, 49, Springer-Verlag, Berlin,
  2004.

\bibitem[Nak04]{Nak04}
N.~Nakayama, \emph{Zariski-decomposition and abundance}, MSJ Memoirs, vol.~14,
  Mathematical Society of Japan, Tokyo, 2004.

\bibitem[Rei80]{Rei80}
M.~Reid, \emph{Canonical 3-folds}, Journ\'ees de {G}\'eom\'etrie {A}lg\'ebrique
  d'Angers (A.~Beauville, ed.), Sijthoof and Nordhoof, Alphen aan den Rijn,
  1980, pp.~273--310.

\bibitem[Sho96]{Sho96}
V.~V. Shokurov, \emph{{$3$}-fold log models}, J. Math. Sci. \textbf{81} (1996),
  no.~3, 2667--2699.

\end{thebibliography}
\end{document}